\documentclass[12pt,reqno]{amsart}
\usepackage{amssymb, amsmath, amsfonts, amsthm, graphics}
\usepackage[hmargin=1 in, vmargin = 1 in]{geometry}
\usepackage{mathrsfs}
\usepackage{newpxtext}
\usepackage{newpxmath}
\usepackage{tikz-cd} 
\usepackage{enumitem}
\setlist{itemsep=0.5em}
\usetikzlibrary{matrix, calc, arrows} 
\usepackage[hyphens]{url}
\usepackage{hyperref}
\usepackage{cleveref}
\usepackage[all]{xy}
\usepackage{marginnote}

\newcommand{\isom}{\cong} 

\theoremstyle{definition}

\numberwithin{equation}{section}

\DeclareMathOperator{\Hom}{Hom}

\DeclareMathOperator{\Spec}{\text{Spec}}

\DeclareMathOperator{\Proj}{\text{Proj}}

\newcommand{\fpt}{\text{fpt}_\fm}

\newcommand{\NN}{\mathbb{N}}

\newcommand{\PP}{\mathbb{P}}
\newcommand{\QQ}{\mathbb{Q}}
\newcommand{\RR}{\mathbb{R}}

\newcommand{\ZZ}{\mathbb{Z}}

\newcommand{\cO}{\mathcal{O}}

\newcommand{\fm}{\mathfrak{m}}

\newcommand{\iL}{\mathscr{L}}
\newcommand{\iF}{\mathscr{F}}
\newcommand{\Fe}{F^{e}_{*}}

\theoremstyle{plain}
\newtheorem{thm}{Theorem}[section]
\newtheorem{Pn}[thm]{Proposition}
\newtheorem{Cor}[thm]{Corollary}
\newtheorem{lem}[thm]{Lemma}
\newtheorem*{mainthm}{Main Theorem}

\theoremstyle{definition}
\newtheorem{dfn}[thm]{Definition}

\newtheorem{eg}[thm]{Example}
\newtheorem{rem}[thm]{Remark}

\newtheorem{notation}[thm]{Notation}
\newtheorem{conjecture}[thm]{Conjecture}

	\title{The F-pure threshold versus the $\text{a}$-invariant for standard graded rings}
	\author{Suchitra Pande}
	\address[S.~Pande]{Department of Mathematics\\University of Utah\\Salt Lake City, 
		UT, 84112\\USA}
\email{\href{mailto:suchitra.pande@utah.edu}{suchitra.pande@utah.edu}}

\begin{document}

\begin{abstract}
    Hirose, Watanabe and Yoshida conjectured a criterion for a standard graded strongly $F$-regular ring to be Gorenstein in terms of the $F$-pure threshold. We complete the proof of this conjecture. We also prove natural extensions of the conjecture to section rings of normal, $F$-split projective varieties with respect to globally generated ample divisors. Our proof exploits the geometry of the Proj of the graded ring.  
\end{abstract}

	\maketitle

\section{Introduction}

The \emph{Frobenius} map provides a powerful tool to capture the singularity properties of prime characteristic rings. Using its properties, Hochster and Huneke introduced the class of \emph{strongly $F$-regular} rings (\Cref{def:FpureFreg}) as a class of mild singularities that are particularly well-behaved \cite{HochsterHunekeTightClosureAndStrongFRegularity}. For a strongly $F$-regular local (or $\NN$-graded) ring $(S, \fm)$ of prime characteristic $p>0$, Takagi and Watanabe introduced a numerical invariant of $S$ called the \emph{$F$-pure threshold} (\Cref{dfn:fpt}) measuring asymptotic properties of the Frobenius splittings of $S$ (\cite{TakagiWatanabeFPureThresh}). The $F$-pure threshold is a positive characteristic analog of the log canonical threshold for complex singularities and which plays a central role in the minimal model program (\cite{HaraYoshidaGeneralizationOfTightClosure}).

    The singularity notions defined via the Frobenius map interact in interesting ways with classical notions such as normal, Cohen-Macaulay, Gorenstein, etc. For instance, strongly $F$-regular rings are always normal and Cohen-Macaulay \cite{HochsterHunekeTightClosureAndStrongFRegularity}. While they are not always Gorenstein, Hirose, Watanabe and Yoshida proposed a conjecture to detect when a strongly $F$-regular standard graded ring $S$ is Gorenstein using the $F$-pure threshold \cite{HiroseWatanabeYoshidaConjecture}. This conjecture involves comparing the value of the $F$-pure threshold with the classical $\text{a}$-invariant of Goto and Watanabe:
    \begin{dfn} \cite{GotoWatanabeOnGradedRings}
        Let $S$ be a $d$-dimensional $\NN$-graded domain with homogeneous maximal ideal $\fm$. Then, the $\text{a}$-invariant of $S$ is defined to be
        \[\text{a}(S) = \min \{ n \in \ZZ \, | \, [H^d _\fm (S) ]_{-n} \neq 0 \}.\]
        Here, $H^d _ \fm (S) = \bigoplus_{n \in \ZZ} [H^d _\fm (S) ]_{n}$ denotes the $d^{\text{th}}$ local cohomology module of $S$ with support in the maximal ideal $\fm$, which we view as a $\ZZ$-graded module over $S$.
    \end{dfn}
    Then the conjecture of Hirose, Watanabe and Yoshida states:
\begin{conjecture} (\cite[Conjecture 1.1]{HiroseWatanabeYoshidaConjecture}) \label{mainconj}
    Let $k$ be an $F$-finite field and $S$ a standard graded algebra over $k$. Assume that $S$ is strongly $F$-regular. Then,
    \begin{enumerate}
        \item [(a)] we always have $\fpt(S) \leq - \text{a}(S)$, where $\fpt(S)$ denotes the $F$-pure threshold of $S$ with respect to $\fm$ (\Cref{dfn:fpt}). 
        \item [(b)] Moreover, $S$ is Gorenstein if and only if $\fpt(S) = -\text{a}(S)$.
    \end{enumerate}
\end{conjecture}

The purpose of this paper is to complete the proof of this conjecture.
\begin{mainthm} (\Cref{mainthmgorcriterion}) \label{mainthmintro}
    Let $S$ be a standard graded, strongly $F$-regular ring over an $F$-finite field of characteristic $p>0$. Then, $S$ is Gorenstein if and only if the negative of the a-invariant of $S$ is equal to the $F$-pure threshold of $S$.
\end{mainthm}

In fact, our result holds more generally for $\NN$-graded section rings of normal, $F$-split projective varieties with respect to globally generated ample divisors (see \Cref{mainthmgorcriterion}).

\medskip

Earlier works towards this include:
\begin{itemize}
    \item The case when $S$ is a toric ring was proved in the original paper \cite{HiroseWatanabeYoshidaConjecture} which inspired the statement of the conjecture.
    \item The case when $S$ is a Hibi ring was proved by Chiba and Matsuda \cite{ChibaMatsudaFpurethresholdforHibirings}.
    \item Part (a) and the only if part of Part (b) of \Cref{mainconj} were both proved by De Stefani and N\'{u}\~{n}ez-Betancourt \cite{DeStefaniNunezBetancourtFthresholdGradedrings}. Their results (with the natural generalizations) also hold for the more general class of standard graded, normal and $F$-pure rings (instead of strongly $F$-regular).
    \item The remaining part of \Cref{mainconj}, which is the if part of Part (b), was proved conditionally by Singh, Takagi and Varbaro in \cite{SinghTakagiVarbaroGorenstienCriterion}, under the additional assumption that the anti-canonical ring of $S$ is finitely generated. We note that this condition is expected to be true for all strongly $F$-regular rings (\cite{DattaSchwedeTuckerFiniteGenerationofMonoidalgebras}), and is known for certain classes such as determinantal rings, strongly $F$-regular rings of dimension at most $3$, and $4$-dimensional strongly $F$-regular rings of characteristic $p>5$. The purpose of this paper is to prove \Cref{mainconj} unconditionally.
\end{itemize}

Our proof is geometric, and proceeds by reinterpreting the $F$-pure threshold and the $a$-invariant in terms of divisor classes on the projective variety $X = \Proj(S)$. This translation is inspired the works of Smith \cite{SmithFujitaFreenessForVeryAmple} and \cite{SmithGloballyFRegular}. The next key idea, following \cite{SchwedeFAdjunction}, is to use duality for the Frobenius map to associate a Weil-divisor $D_\varphi$ on $X$ to any homogeneous $p^{-e}$-linear map $\varphi: \Fe S \to S$. With this association, the proof of \Cref{mainthmgorcriterion} reduces to the observation that on a projective variety over a field $k$, the only divisor class $E$ that is at once effective and a limit of anti-effective classes is the trivial class (\Cref{Weilpseudoeffectivelemma}).

It is natural to seek generalizations of \Cref{mainconj} to other settings such as non-standard graded rings, or local rings. Note that the assumption on normality of $S$ can not be dropped as shown in \cite[Example 5.3]{DeStefaniNunezBetancourtFthresholdGradedrings}. A compelling generalization of \Cref{mainconj} to all normal and $F$-pure local rings was introduced in \cite{SinghTakagiVarbaroGorenstienCriterion} using an invariant called the \emph{$F$-injective} threshold to replace the $\text{a}$-invariant. Since our techniques involve the projective geometry of $\Proj(S)$, they do not seem to extend directly to this setting.

\medskip

\paragraph{\textbf{Acknowledgements.}} I would like to thank Karen Smith for inspiring conversations regarding many of the techniques used in this paper. I would also like to thank Benjamin Baily, Anna Brosowsky, Alessandro De Stefani, Havi Ellers, James Hotchkiss, Seugsu Lee, Devlin Mallory, Karl Schwede, Anurag Singh, Olivia Strahan, Shunsuke Takagi and Kevin Tucker for valuable conversations related to topics in this paper. This material is based upon work supported by the National Science Foundation under Grant No. DMS-1928930 and by the Alfred P. Sloan Foundation under grant G-2021-16778, while the author was in residence at the Simons Laufer Mathematical Sciences Institute (formerly MSRI) in Berkeley, California, during the Spring 2024 semester.

\section{Preliminaries}
\begin{notation} \label{standingnotation}
    Unless stated otherwise, all rings considered will be commutative, Noetherian and with a unit. Throughout, $k$ will denote an $F$-finite field of characteristic $p>0$ (which means that we have $[k:k^p] < \infty $).
\end{notation}

\subsection{The Frobenius map in characteristic $p>0$.}

Let $R$ be any ring of prime characteristic $p$. Then for any $e \geq 1$, let $F^e: R \to R$ sending $r \mapsto r^{p^e}$ be the $e^{\text{th}}$-iterate of the \emph{Frobenius morphism}.  Since $R$ has characteristic $p$, $F^e$ defines a ring homomorphism, allowing us to define a new $R$-module for each $e \geq 1$ obtained via restriction of scalars along $F^e$. We denote this new $R$-module by $F_{*} ^e R$ and its elements by $F_{*} ^e r$ (where $r$ is an element of $R$). Concretely, $F_{*} ^e R$ is the same as $R$ as an abelian group, but the $R$-module action is given by:
 $$ r\cdot F_* ^e s := F_{*} ^e (r^{p^e} s) \textrm{\quad for  $r\in R$ and $F_* ^e s \in F_* ^e R$}   .$$

The ring $R$ is called \emph{$F$-finite} if the Frobenius map $F$ on $R$ is a finite map.

\subsection{F-purity and F-regularity}

\begin{dfn}\cite{HochsterHunekeTightClosureAndStrongFRegularity} \label{def:FpureFreg}
      Let $R$ be a Noetherian $F$-finite ring of characteristic $p$. $R$ is said to be \emph{$F$-pure} if the map $R \to F_* R$ splits as a map of $R$-modules. And $R$ is said to be \emph{strongly $F$-regular} if for any element $c \in R$ that is not contained in any minimal prime of $R$, there exists an integer $e \gg 0$, such that, the following map
      \begin{equation*}
      \begin{split}
          R \to \Fe R \\
          1 \mapsto \Fe c
      \end{split}
      \end{equation*}   
      splits as a map of $R$-modules.
\end{dfn}

\begin{rem}
    By \cite{HochsterHunekeFRegularityTestElementsBaseChange}, a strongly $F$-regular ring is always, normal and Cohen-Macaulay.
\end{rem}

\subsection{The F-pure threshold}
Let $(R, \fm)$ be an $F$-finite local domain over a field $k$ of characteristic $p>0$.
\begin{dfn}[\cite{TakagiWatanabeFPureThresh}] \label{dfn:fpt}
    Assume that $R$ is $F$-pure. For each $e \geq 1$, let
    \[ \nu_e (\fm) = \max \{ n \, | \, \exists \, x \in \fm^n \, \text{such that the map } R \to \Fe R, \,  1 \mapsto \Fe x \, \text{ splits} \}. \]
    The $F$-pure threshold of $(R, \fm)$ with respect to the maximal ideal $\fm$ is defined to be the limit
    \[ \fpt(R) := \lim_{e \to \infty} \frac{\nu_e (\fm)}{p^e}. \]
\end{dfn}

\subsection{Section rings and modules}

 		\begin{dfn} \label{sectionringdfn}
		Let $k$ be a field and $X$ be a projective scheme over $k$. Given an ample invertible sheaf $\iL$ on $X$, the $\NN$-graded ring $S$ defined by
		$$ S =	S(X, \iL) :=  \bigoplus _{n \geq 0} H^{0}(X, \iL^{n})	$$
		is called the \emph{section ring} of $X$ with respect to $\iL$. The affine scheme $\Spec(S)$ is called the (affine) \emph{cone over $X$} with respect to $\iL$.

        Similarly, given a coherent sheaf $\iF$ on $X$, the $\ZZ$-graded module defined by
        \[ M(X, \iL) :=  \bigoplus _{n \in \ZZ} H^{0}(X, \iF \otimes \iL^{n})	  \]
    is called the \emph{section module} of $\iF$ with respect to $\iL$.
	\end{dfn}
\medskip

    \begin{eg}
        Let $X \subset \PP^N$ be a normal, projective subvariety of $\PP^N _k$. Then, the section ring of $X$ with respect to the very ample line bundle $\cO_X (1)$ is the normalization of the homogeneous coordinate ring $k[x_0, \dots, x_N]/I_X$ (where the $I_X$ is the homogeneous radical ideal corresponding to $X$). Note that even though the homogeneous coordinate ring is always standard graded (i.e., generated by degree one elements), the corresponding section ring, being the normalization, might not be standard graded.
    \end{eg}

In this paper, we will consider a slightly more general situation than the above example, namely with respect to globally generated invertible sheaves. Recall that an ample invertible sheaf $\iL$ is globally generated if and only if the corresponding section ring $S(X, \iL)$ admits a homogeneous system of parameters consisting of degree $1$ elements (see \cite[Section 1]{SmithFujitaFreenessForVeryAmple}).

\begin{dfn} \label{dfn:Ie}
    Suppose $S$ is a section ring over an $F$-finite field $k$ of characteristic $p>0$. Assume that $X = \Proj(S)$  is normal. Let $L$ denote an ample divisor corresponding to $\cO_X (1)$. Let $I_e (mL)$ denote the \emph{$e^{th}$-splitting subspace} of $S_m$ on $X$ defined as:
\[I_e (mL) = \{ x \in S_m \, | \, \varphi (\Fe x) = 0 \,  \forall \, \varphi \in \Hom_{\cO_X} (\Fe \cO_X(mL), \cO_X) \}. \]
\end{dfn}
Recall that $I_e (S) = \bigoplus_{m \geq 0} I_e (mL)$ is the splitting ideal of $S$. 

\begin{lem} \label{fptlimit}
    Suppose $S$ is the section ring of a normal projective variety $X$ (over $k$) with respect to a globally generated ample invertible sheaf $\iL$. Let $\fm \subset S$ denote the homogeneous maximal ideal, then the $F$-pure threshold of $S$ can be computed as follows:
    let
    \[  \nu'_e(S) := \max \{ r \, | \, S_{ \geq r} \not \subset I_e (S) \} = \max \{ r \, | \, I_e (\iL^r) \neq S_r \} . \]
    Then,
    \[ \fpt (S) = \lim _{e \to \infty} \, \frac{\nu'_e(S)}{p^e -1}.\]
\end{lem}

\begin{proof}
    When $S$ is standard graded, this is just a restatement of \cite[Proposition~2.5]{SinghTakagiVarbaroGorenstienCriterion}, since $\fm^r = S_{\geq r}$. In the more general case, since $S$ is $F$-pure, we know that the projective variety $X$ is globally $F$-split. Then, as $\iL$ is ample we have the following cohomology vanishings by \cite[Proposition 1]{MehtaRamanathanFrobeniusSplittingAndCohomologyVanishing}:
    \[ H^i (X, \iL^m) = 0, \quad \forall \, i >0, \, m >0.\]
    In particular, since $\iL$ is globally generated, the sheaf $\iL^d$ is $0$-regular with respect to $\iL$, where $d = \dim (S) =  \dim(X) +1 $.  In particular, the multiplication maps 
    \begin{equation} \label{eqn:surjection} S_r \times S_m \twoheadrightarrow S_{m+r}\end{equation} are surjective whenever $r \geq 0$ and $m\geq d$. Therefore, for all $r \geq d$, we have inclusions
    \[ \fm^r \subset S_{\geq r} \subset \fm^{r -d}.  \]
    We refer to \cite[Section 1.8]{LazarsfeldPositivity1} for details regarding the use of Castelnuovo-Mumford regularity.

    With notation as in \Cref{dfn:fpt}, this shows that we have
    $\nu_e \leq \nu'_e \leq \nu_e + d$ for all $e$. Dividing by $p^e -1$ and taking a limit as $e \to \infty$, we obtain $\fpt(S ) = \lim_{e \to \infty} \frac{\nu' _e}{p^e -1}$.
    \end{proof}

\begin{Cor}
    Let $S$ be as in \Cref{fptlimit}. Then, for any $n \geq 1$ we have
    \[ \fpt(S^{(n)}) = \frac{\fpt(S)}{n}, \]
    where $S^{(n)}$ denotes the $n^{\rm{th}}$-Veronese subring of $S$.
\end{Cor}
\begin{proof}
    With notation as in \Cref{fptlimit}, for any $e \geq 1$, we claim that
    \[ \nu'_e (S^{(n)})+1 + \frac{d}{n} \geq \frac{\nu'_e (S)}{n} \geq \nu'_e (S ^{(n)}) .\]
    The right-inequality follows immediately from \Cref{dfn:Ie}. For the left inequality, we note that by the definition of $\nu'_e$, for all $m \geq \nu_e '+1 $ we have $I_{e} (\iL^{mn}) = S_{mn}$. Suppose $I_e (\iL^r) \neq S_r$ for some $r \geq n(\nu _e ' +1 ) + d$. Let $f \in S_r \setminus I_e (\iL^r)$. Then, by the surjection \Cref{eqn:surjection}, we may write $f = \sum g_i h_i$, where $g_i \in S_{n(\nu_e ' + 1)}$. But this is a contradiction since $I_{e} (\iL^{n(\nu_e ' +1 )}) = S_{n (\nu_e ' + 1)}$ and $I_e$ is an ideal of $S$.  This proves the claim. The Corollary now follows by diving by $p^e$ and taking a limit as $e \to \infty $.
\end{proof}

\section{Proof of Main Theorem}

In this section, we will prove the main theorem of this paper in \Cref{mainthmgorcriterion} as conjectured in \cite{HiroseWatanabeYoshidaConjecture}. In fact, we will prove a slightly more general result that works for normal, $F$-pure section rings that admit a system of parameters by elements of degree $1$.

\begin{thm} \label{mainthmgorcriterion}
    Let $S$ be an $\NN$-graded section ring over an $F$-finite field $k$ of characteristic $p>0$ (\Cref{sectionringdfn}). Assume that $S$ is normal, $F$-pure and admits a homogeneous system of parameters of degree one elements. Also suppose that the dimension $d+1$ of $S$ is at least two and set $a : = -\text{a}(S)$. Then, we have
    $\fpt(S) \leq a$  and equality holds if and only if $S$ is quasi-Gorenstein.
\end{thm}

\begin{dfn}
    Recall that the $a$-invariant of an $\NN$-graded ring $S$ (with homogeneous maximal ideal $\fm$) is defined by
\[-\text{a}(S) = \min \{ r \, | \, [H^{d+1} _\fm (S)]_{-r}  \neq 0 \} \]
where $d+1$ is the dimension of $S$.
\end{dfn} 

\begin{dfn}
    A normal local ring $R$ with a canonical module $\omega_R$ is called quasi-Gorenstein if $\omega_R \isom R.$ 
\end{dfn}

Our proof of \Cref{mainthmgorcriterion} uses the projective geometry of $\Proj (S)$. We first review some constructions that are helpful for translating between the algebraic and geometric interpretations.

\subsection{Graded canonical modules.} \label{gradedcanonicalmodule}
For an $\NN$-graded section ring $S$ of a normal projective variety $X$ with respect to ample invertible sheaf $\iL = \cO_X (1)$ (\Cref{sectionringdfn}), there is a natural choice of a graded module isomorphic to a canonical module of $S$.
Since $X$ is normal, we may consider the canonical divisor $K_X$ of $X$ by extending from the regular locus of $X$. Moreover, the reflexive sheaf $\cO_X(K_X)$ is isomorphic to the canonical sheaf of $X$. Then the section module $M(\cO_X (K_X), \iL)$ (\Cref{sectionringdfn}) is a reflexive module of rank $1$ over $S$ that is isomorphic to a canonical module of $S$. See \cite[Section 1]{SmithFujitaFreenessForVeryAmple} for a detailed discussion regarding the graded canonical module of a section ring.

Since any graded module that is isomorphic to $S_\fm$ after localization at $\fm$ must be isomorphic to a shift $S(r)$ for some $r >0$, $S$ being quasi-Gorenstein is thus equivalent to the condition that $K_X \sim r\iL$ as Weil-divisors on $X$, for some $r \in \ZZ$. 

\subsection{Homogeneous Splittings:} \label{homogenoussplittings} We will use the following notion of ``homogeneous splittings" over $S$: Since $S$ is an $\NN$-graded ring, for each $e \geq 1$, the $S$-algebra $\Fe S$ is naturally an $\frac{1}{p^e} \NN$-graded such that $F^e$ is a degree-preserving map. Moreover, for any $\ZZ$-graded $S$-module $N$, the module $ \Hom _S (\Fe S , N)$ is $\frac{1}{p^e}\ZZ$-graded with a decomposition given as
\begin{equation} \label{decompositionofHom} \Hom _S (\Fe S , N) =\bigoplus_{\frac{n}{p^e}\in \frac{1}{p^e}\ZZ}\Hom_S^{\text{gr}}\left(F^e_*S\left[\frac{-n}{p^e}\right], N\right) \end{equation}
where $F^e_*S\left[\frac{-n}{p^e}\right]$ is the module whose $\frac{m}{p^e}$ degree component (for any $m \in \ZZ$) is $(F^e_*S)_{\frac{m-n}{p^e}}$. Here, $\Hom_S^{\text{gr}}\left(F^e_*S\left[\frac{-n}{p^e}\right], N\right)$ is the space of degree-shifting maps $\Fe S \to N $ that send degree $\frac{- n + m p^e}{p^e}$ elements of $\Fe S$ to degree $m$ elements of $N$ and are zero on the remaining degrees.

This perspective is useful for geometric interpretations as follows: note that the $S$-module 
\[ M_{e, n} := \bigoplus _{m \in \ZZ} (F^e_*S)_{\frac{n + mp^e}{p^e}}  \]
is the section module (\Cref{sectionringdfn}) of the sheaf $\Fe \cO_X (nL)$, where $X = \Proj (S)$ with $L$ corresponding to the invertible sheaf $\cO_X (1)$. Thus, if we have a homogeneous element $f \in S_n$ of degree $n$ such that $f \notin I_e (S)$ (where $I_e$ is as in \Cref{dfn:Ie}), then there exists a homogeneous splitting $\varphi: \Fe S \to S$ such that $\varphi (\Fe f) = 1$, which corresponds to an element $\varphi \in \Hom_S^{\text{gr}}\left(F^e_*S\left[\frac{n}{p^e}\right], S\right) $. This homogeneous then induces a splitting $\tilde{\varphi}: \Fe \cO_X (D_f) \to \cO_X $, where $D_f$ denotes the divisor of zeroes of $f$. Thus, homogeneous splittings exactly correspond to splittings coming from sheaf maps on $X$. For a detailed discussion regarding this, we refer to \cite[Section~2]{LeePandeFsignaturefunction}.

\medskip

\begin{proof}[Proof of \Cref{mainthmgorcriterion}]
    First, we may reduce to the case of an algebraically closed base field $k$ by taking a flat base change $k \to \overline{k}.$ The $\text{a}$-invariant is preserved by flat base change for local cohomology, and the $F$-pure threshold of $S$ is preserved since the base change map $S \to S\otimes_k \overline{k}$ has regular fibers (see \cite[Theorem 5.6]{YaoObservationsAboutTheFSignature}). So we assume that $k = \overline{k}$ is algebraically closed.
    
    Now, we reinterpret all the terms involved in terms of $X: = \Proj (S)$ and $L$ be a Cartier divisor corresponding to $ \cO_X(1)$. Note that since $S$ is assumed to be normal, $X$ is normal. Recall that the local cohomology module $H^{d+1} _ \fm (S)$ is isomorphic to the graded module (since $d \geq 1$) 
    \[H^{d+1} _\fm (S) = \bigoplus _{r \in \ZZ} H^d (X, \cO_X(rL)). \] Therefore, the negative $\text{a}$-invariant of $S$ can be reinterpreted as $a = \min \{ r \, | \, [H^{d+1} _\fm (S)]_{-r} = H^d (X, \cO_X(-rL)) \neq 0 \}$.
    
    Furthermore, we consider the canonical divisor $K_X$ (by extending from the regular locus) of $X$, which gives us the corresponding reflexive sheaf $\cO_X (K_X)$ isomorphic to the canonical sheaf of $X$. In particular, by duality for the canonical sheaf (see \cite[Section 1]{SmithFujitaFreenessForVeryAmple} for a detailed discussion), we may rewrite: $ H^d (X, \cO_X(-rL))\isom H^0 (X, \cO_X(K_X + rL)).$ Therefore, the negative $\text{a}$-invariant is reformulated as
    \[ a = \min \{ r \, | \, H^0 (X, \cO_X(K_X + rL)) \neq 0 \}.  \]
    
    Therefore, by definition, we have $H^0 (X, K_X + aL) \neq 0$, which defines an effective Weil-divisor $D \geq 0$ such that
    \begin{equation} \label{eqn:DaL} D \sim K_X + aL.\end{equation}
    Thus, by the discussion in \Cref{gradedcanonicalmodule}, $S$ will be quasi-Gorenstein if we can show that $D = 0$ as a Weil divisor.

    Now, suppose we have a homogeneous element $f \in S$ of degree $m$ and let $D_f$ denote the Cartier divisor of zeroes corresponding to the section $f$ of $\cO_X (n)$. Then by the above discussion on homogeneous splittings, whenever $f \notin I_e (nL)$ (\Cref{dfn:Ie}) we obtain a homogeneous splitting $\varphi$ of $f$ giving us a global splitting 
    \[ \tilde{\varphi}: \Fe \cO_X(D_f) \to \cO_X  \]
    as $\cO_X$-modules.
    By the duality isomorphism
    \[ \mathscr{H}om _{\cO_X} (\Fe \cO_X(D_f) , \cO_X) \isom \Fe \cO_X ((1-p^e)K_X - D_f), \]
    for the Frobenius map, $\tilde{\varphi}$ in turn gives us a dual effective Weil divisor $D_\varphi \sim (1-p^e) K_X - D_f$.
     See \cite[Section~4.1]{SchwedeSmithLogFanoVsGloballyFRegular} for a detailed discussion regarding duality for the Frobenius map.

     Then, as $\QQ$-Weil divisors, we have the linear equivalence
    \[ \frac{1}{p^e -1} D_\varphi \sim_\QQ  -K_X - \frac{n}{p^e -1} L,\]
    since $D_f \sim nL$.
    Combining with \Cref{eqn:DaL}, we obtain
    \[ D+ \frac{1}{p^e -1} D_\varphi \sim_\QQ (a - \frac{n}{p^e -1})L.\]
    Now, since both $D$ and $D_\varphi$ are effective and $L$ is ample, this implies that
    \[ a - \frac{n}{p^e -1} \geq 0 \implies  \frac{n}{p^e -1} \leq a .\]
    Therefore, in the notation of \Cref{fptlimit}, we must have
    \[\frac{\nu'_e (S)}{p^e -1} \leq a  \]
    for each $e$ (since by definition, for each $e$, there is a non-zero element $f_e$ of degree $\nu'_e$ in $S$ such that $f_e \notin I_e (S)$).  Hence, by \Cref{fptlimit} we must have \[ \fpt (S) \leq a.\]
    \medskip
    
    Now, suppose $\fpt(S) = a$. Then, we may pick a sequence of homogeneous elements $f_e \notin I_e (n_e L)$ of degree $n_e$ of $S$ such that $\lim_{e \to \infty} \frac{n_e}{p^e -1} = a$. Then, by the previous discussion, we can construct effective Weil-divisors $D_e$ (corresponding to the splittings $\varphi_e$ of $D_{f_e}$) such that
    \begin{equation} \label{limitzero} D + \frac{1}{p^e - 1} D_e \sim_{\QQ} (a - \frac{n_e}{p^e -1}) L.\end{equation}
    We now claim that this forces $D$ to be the zero divisor. Assuming this claim for the moment, we note that we are done since this forces $-K_X \sim aL$ and in turn, $S$ is quasi-Gorenstein.

    Thus, it remains to show that \Cref{limitzero} forces $D$ to be the zero divisor. First, we assume that $K_X$ is $\QQ$-Cartier. Note that this implies that $D$ and $D_e$ are all $\QQ$-Cartier divisors. Now, we consider the classes of $D$ and $D_e$ in the N\'eron-Severi space $N^1 _\RR (X)$, denoted by $[D]$ and $[D_e]$. \Cref{limitzero} implies that in $N^1 _\RR (X)$,
    \[ \lim_{e \to \infty} [D+D_e] = \lim _{e \to \infty} [D]+ [D_e] = \lim _{e \to\infty} (a - \frac{m_e}{p^e -1}) [L] = 0.  \]
    Thus, we see that the class $[D]$ is the limit of $-[D_e]$, thus, $-[D]$ is pseudo-effective. But, $D$ was effective by assumption. Since there are no lines contained in the pseudo-effective cone of $X$, $[D]$ is forced to be zero (see \cite[Lemma 2.3]{CasciniHaconMustataSchwedeNumDimPseudoEffective}).
    Now, let $r \geq 0$ be such that $rD$ is Cartier. Then, since $[rD] = r[D]$ is zero, we must have that $rD$ is the zero divisor, because if not, for a general curve $C \subset X$ that intersects $D$ but is not contained in $D$, we must have that the degree of $rD|_C$ is positive, since it is effective. This completes the proof if $K_X$ is $\QQ$-Cartier.

    In the general case, we need a more careful version of the above argument, which is presented in \Cref{Weilpseudoeffectivelemma} and shows that in this case again, $D$ is forced to be zero which concludes the proof of the theorem.
\end{proof}

The following lemma was used in the proof above:
\begin{lem} \label{Weilpseudoeffectivelemma}
    Let $X$ be a normal projective variety and $L$ be an ample divisor on $X$. Suppose we have an effective Weil-divisor $D$, a sequence of effective $\QQ$-Weil divisors $(D_e)_{e \geq 1}$ and a sequence of non-negative rational numbers $ \lambda_e$ such that
    \begin{equation} \label{lemassumptionone} D+D_e \sim_{\QQ} \lambda_e L \end{equation}
    and
    \begin{equation} \label{lemassumptiontwo} \lim _{ e \to \infty} \lambda_e = 0.  \end{equation}
    Then, we necessarily have $D = 0$.
\end{lem}

To prove this, we will use intersection theory of divisors.

\subsection*{Numerical Intersection Theory}
	
	We review Kleiman's numerical intersection theory of divisors as developed in \cite{Kleimannumericaltheory}. We only state the definition and main facts that we need here.
	
	Let $Y$ be a proper scheme over a field $K$ of dimension $d$. Let $\iL_{1}, \dots, \iL_{r}$ be line bundles on $Y$, $D$ an effective Cartier divisor and $\iF$ a coherent sheaf on $Y$. Then we have the following theorem due to Snapper.
	
	\begin{thm}{\cite{Kleimannumericaltheory}} \label{Snapper}
		Consider the function $f(m_{1}, \dots, m_{r}) = \chi (\iF \otimes \iL_{1} ^{m_{1}} \otimes  \dots \otimes \iL_{r}^{m_{r}})$ for $m_{1}, \dots, m_{r} \in \ZZ$, where $\chi$ denotes the Euler characteristic (over $K$). Then there is a polynomial $P(x_{1}, \dots, x_{r})$ with coefficients in $\QQ$ and of total degree $\leq s = \dim \mathrm{Supp}(\iF)$ such that $P(m_{1}, \dots, m_{r}) = f(m_{1}, \dots, m_{r})$ for all $m_{1}, \dots, m_{r} \in \ZZ$.
	\end{thm}
	
	\begin{dfn} {\cite[Section 2, Definition 1]{Kleimannumericaltheory}}
		Suppose that $\dim Y \leq r $, then we define the intersection number $ 	(\iF; \iL_{1} \cdot \ldots \cdot \iL_{r})$ to be the coefficient of the term $x_{1} \cdots x_{r}$ in the polynomial $P(x_{1}, \dots, x_{r})$ as in Theorem \ref{Snapper}.
	\end{dfn}
	
	For any coherent sheaf $\iF$ with $\dim (\mathrm{Supp}(\iF)) \leq r$, this defines an integer valued multi-linear form on Pic$(Y)^{r}$. If a line bundle $\iL$ is defined by a Cartier divisor $D$, then we sometimes write $D$ instead of $\iL$ in the intersection form. We need the following properties of intersection numbers:
	
	\begin{Pn} \label{interpn}
		\begin{enumerate}

			\item (\cite[Section 2, Proposition 4]{Kleimannumericaltheory}) \label{intersection1} If $\iF$ is a locally free sheaf, and $D$ is a Cartier divisor on $X$, then $$(\iF; \iL_{1} \cdot \ldots \cdot \iL_{d-1} \cdot \cO_{Y}(D )) =  (\iF|_{D} ; \iL_{1}|_{D} \cdot \ldots \cdot \iL_{d-1}|_{D}) $$

			\item (\cite[Section 2, Proposition 5, Corollary 1,2]{Kleimannumericaltheory}) \label{intersection4} Let $V = \mathrm{Supp}(\iF)$ and $V_{1}, \dots, V_{s}$ its irreducible components. Let $l_{i} = \mathrm{length}(\iF \otimes \cO_{V_{i}} )$ where $\cO_{V_{i}}$ is the stalk of $\cO_{Y}$ at the generic point of $V_{i}$, then, if $r \geq \dim(V)$
			$$ (\iF;\iL_{1} \cdot \ldots \cdot \iL_{r}) = \sum _{i=1} ^{s} l_{i} \times (\iL_{1}|_{V_{i}} \cdot \ldots \cdot \iL_{r}|_{V_{i}})		$$
			In particular, if $\iF$ is an invertible sheaf, then $(\iF; \iL_{1} \cdot \ldots \cdot \iL_{d}) = (\iL_{1} \cdot \ldots \cdot \iL_{d})$.
		\end{enumerate}
		
	\end{Pn}
	
\begin{proof}[Proof of \Cref{Weilpseudoeffectivelemma}]
    First we note that if $\lambda_e = 0$ for any $e$ then $D + D_e \sim_{\QQ} 0$. Since both $D$ and $D_e$ are effective and $X$ is projective, this implies that $D = D_e = 0$ as required. Therefore, we may assume that $\lambda_e $ is positive for all $e$.
    
    For each $e \geq 1$, write $\lambda _e = \frac{a_e}{b_e}$ for postive integers $a_e$ and $b_e$. Moreover, by replacing both $a_e $ and $b_e$ by a positive multiple, we may assume that $b_e D$ and $b_e D_e$ are both $\ZZ$-Weil divisors. Then, since $D + D_e \sim_\QQ \lambda_e L$, we may assume that $b_e(D+ D_e) \sim a_e L$. In particular, $Z_e := b_e (D+ D_e)$ is an effective Cartier divisor on $X$. Now, using Part (1) of \Cref{interpn} we see that 
    \begin{equation} \label{intersectionequality}
       a_e (L \cdot \ldots \cdot L \cdot L) = (L \cdot \ldots \cdot L \cdot \cO_{Y}(Z_e )) =   ( L|_{Z_e} \cdot \ldots \cdot L|_{Z_e}).
    \end{equation}
    Now suppose $D \neq 0$, then let $E$ be a non-zero prime divisor in the support of $D$. Then, we have $Z_e \geq b_e E$ for each $e \geq 1$. Moreover, since $L$ is ample, the restriction of $L$ to any prime divisor of $X$ has a positive intersection number. Using this observation and Part (2) of \Cref{interpn}, we have
    \begin{equation} \label{intersectioninequality}
        ( L|_{Z_e} \cdot \ldots \cdot L|_{Z_e}) \geq b_e (L|_E \cdot \ldots \cdot L|_E).
    \end{equation}
   Now, \Cref{intersectionequality} and \Cref{intersectioninequality} together imply that $\lim_{e \to \infty} \frac{a_e}{b_e} \geq  \frac{(L|_E \cdot \ldots \cdot L|_E)}{(L \cdot \ldots \cdot L \cdot L)} > 0$, which is a contradiction to the assumption that $\lim_{e \to \infty} \lambda_e = 0$. Thus we must have $D = 0$, which completes the proof. 
\end{proof}

 \bibliographystyle{alpha}
    \bibliography{Main}
\end{document}